\documentclass[11pt,letterpaper]{article}
\usepackage[margin=1in]{geometry}

\input{style}

\title{A Framework for Blockchain Architecture Design}
\author{Partha S.~Dey \and Aditya S.~Gopalan}
\date{}

\begin{document}
\maketitle

\begin{abstract}
Emerging applications of blockchains, such as grocery supply chains, require frequent updates to the data structure.
This is in contrast with typical analyses of the Bitcoin blockchain, in which updates occur infrequently.
With more frequent updates, the spread of blocks among participants in the blockchain protocol becomes complicated; thus, the structure of the blockchain data structure itself can differ significantly from the structure without the presence of network delays. In addition, emerging blockchain applications such as internet-of-things or supply chain warrant different architectures of the blockchain data structure and so one needs a general understanding of how the data structure works rather than focusing on the specific architecture of Bitcoin.
In this paper, we develop a new model to study the dynamics of the blockchain data structure in the presence of i.i.d.~network delays.
Specifically, we consider an asymptotic design criterion called \emph{one-endedness} which should be satisfied by all blockchain architectures.
We develop techniques to show that the one-endedness property holds for some of the leading blockchain architectures.
\end{abstract}

\section{Introduction}\label{sec:introduction}

Recent interest in blockchains has come from a variety of applications, such as monetary currencies, internet of things, and supply chains.
In all of these applications, \emph{consensus} between the participating parties is the key principle.
It is natural that as blockchain technology is developed further, different architectures and structural properties may be required for different applications.
Nevertheless, a unifying principle of how blockchains achieve consensus would allow for the development of general frameworks for operating and design of blockchains.
In this paper, we study one such asymptotic principle called \emph{one-endedness} of a limiting (in time) blockchain and verify that one-endedness holds for the Bitcoin and Iota blockchains, which are two fundamentally different blockchain architectures.
Both can be analyzed using similar techniques that we develop in this paper.

The Bitcoin blockchain is the canonical blockchain for monetary currencies, and the Iota blockchain is a leading blockchain for internet-of-things.
Despite significant recent interest in applying blockchains to supply chains~(\cite{keskin2024blockchain, iyengar2023economics, iyengar2024blockchain, cui2024supply}), no supply-chain-specific blockchain architecture has been proposed. The techniques and unified framework developed in this paper to verify one-endedness provide an important guideline for the development of blockchain architectures for supply chains and other applications.

\subsection{Blockchain Background}
Blockchain protocols comprise of two parts: a peer-to-peer communication network and a directed acyclic graph called the \emph{blockchain data structure}.
We discuss these separately.

Nodes in the peer-to-peer network create blocks and subsequently communicate these blocks to all other nodes.
These blocks contain lists of application-specific information within their contents, but in this paper we will treat blocks as atomic units.
In the Bitcoin blockchain, the creation of blocks is very slow compared to their propagation - in queueing language, this corresponds to sending the arrival rate to zero.

However, in emerging applications of blockchains, such as grocery supply chains, frequent and granular information is more useful than longer lists of information which are updated more infrequently~(\cite{keskin2024blockchain}).
Thus, one also needs to have analyses of blockchain operation in more traditional regimes where the arrival rate is not close to zero.

Recent papers by~\cite{gopalan2020stability} and~\cite{gopalan2021data} investigate the performance of the peer-to-peer network as it is implemented in Bitcoin.
Viewing the peer-to-peer network as the server in a queue, they derive stability conditions for the block creation rate.
They also study the propagation of blocks in the peer-to-peer network.
Different blocks propagate in a highly non-independent way.
In this more general context, the traditional assumptions made in Bitcoin analyses that either (1) one block is fully propagated across the network before the next is created, or (2) block propagation times are uniformly bounded, both fail to hold.

The blockchain data structure is comprised of the blocks created by the nodes in the peer-to-peer network.
Specifically, the blockchain data structure is a graph where the blocks are vertices.
Each block contains (at least) one \emph{reference}, or directed edge, to one (or more) previous block(s).
It is easy to see that the blockchain data structure is a directed acyclic graph.
Throughout this paper, we treat directed edges as moving \emph{from older blocks to newer blocks}.
A blockchain's architecture is a pre-specified \emph{attachment rule} $f$ to determine how new blocks should attach to the existing blockchain.
Let $G_t := G_t(f)$ be the blockchain data structure comprising of the first $t+1$ blocks.
For mathematical convenience, we treat $G_0$ as a single root vertex with no edges, and we assume that this root vertex is also the root of $G_t$ for all $t > 0$.
Due to the complicated block propagation behavior exhibited in~\cite{gopalan2020stability, gopalan2021data}, it is generally \emph{not} true that $G_{t+1} = f(G_t)$.
This can be seen, for example, when one node creates a new block $b$, and another node creates a new block before hearing of block $b$ via peer-to-peer network communication.
This behavior is especially relevant in settings like grocery supply chains, where updates need to occur frequently.
Any unified model for blockchain analysis, design, or operation, must account for this challenge with an explicit characterization of network delay.

\subsection{Model Overview}
In this paper, we introduce a new model which isolates the dynamics of the blockchain data structure while also incorporating the effect of network delay.
Our model evolves in discrete time and may be taken by, \textit{e.g.}, discretizing a continuous-time block arrival process on the epochs of the arrivals.
Specifically, we will consider dynamics of the form 
\begin{equation}
    G_{t+1} = f(G_t, G_{t - \xi_t}),
    \label{eq:acm-recursion}
\end{equation}
where $(\xi_t)_t$ is a sequence of i.i.d. \emph{network delays}.
We refer to Equation~(\ref{eq:acm-recursion}) as an \emph{asynchronous recursion} to highlight the effect of the network delay.
When $\xi_t = 1$ for some time $t$, we have that $G_{t+1} = f(G_t)$.
We use the term \emph{synchronous} to specifically refer to the setting where $\pr(\xi_1 = 1) = 1$, and the term \emph{asynchronous} to specifically refer to the setting where $\pr(\xi_1 = 1) < 1$.

We will sometimes use $G_t(f)$ instead of $G_t$ to denote that the model is with respect to attachment function $f$ where it helps readability.
The details required to carefully specify the model, as well as the graph limits as $t \to \infty$, are technical and deferred to Section~\ref{sec:defs-main}.

\subsection{Consensus in Blockchains}
Consensus in blockchains is determined on a block-by-block basis.
In this paper, we say that a block is \emph{confirmed} if it can be determined from the blockchain data structure that consensus has occurred for that block.

Due to its prominence in the literature, we begin by discussing consensus and block confirmation in Bitcoin.
Bitcoin's attachment function $f_\mathrm{Nak}$ creates a tree: each arriving block attaches to a single leaf in the tree of maximal hop distance from the root, with any ties broken uniformly at random.
The implicit idea in the Bitcoin whitepaper, as well as most analyses of Bitcoin, is that there should only be a single confirmed block at any given hop distance $h$ from the root.
Since the arrival of new blocks in Bitcoin is very slow compared to block propagation, heuristics such as ``treat a block $b$ as confirmed if there exists a path of length at least $k$ from $b$ some leaf'' are often used in practice.
This heuristic is often used as the definition of confirmation (see, \textit{e.g.}, its use in~\cite{guo2022bitcoin}).

To this point, we make three important remarks which inspire the alternative approach taken in this paper.
First, the heuristic is an approximation of the idea that in the limit as $t \to \infty$, only one block at each hop distance $h$ from the root should have infinitely many children.
This is the block that should be confirmed.
Second, using $f_\mathrm{Nak}$ in the asynchronous setting of our model, it will frequently happen that a block achieves the heuristic (for any fixed $k$), but nevertheless only has finitely many children in the limiting graph.
Under the heuristic, taking the first and second points together leads to contradictory notions.
Third, this heuristic is specific to $f_\mathrm{Nak}$ and it is difficult to generalize it to any other attachment rule $f$.

These suggest that one should take an asymptotic (as $t \to \infty$) approach to determining which blocks are confirmed.
One such approach, which we adopt in this work, is that of~\cite{gopalan2020stability}.
The key to their approach is to take a subtly different view of the limiting heuristic for Bitcoin's block confirmation: any confirmed block should have a directed path to all but finitely many newer blocks in the blockchain.
Since Bitcoin's blockchain is a tree, these two ideas are equivalent for Bitcoin; they are generally not equivalent for an arbitrary attachment rule $f$.

Treating a confirmed block as one that has a directed path to all but finitely many newer blocks in the blockchain,~\cite{gopalan2020stability} use the idea of \emph{one-endedness} of the limiting blockchain graph, which is a technical condition implying the existence of infinitely many confirmed blocks.
They establish the one-endedness property for two simple blockchain attachment rules, including Bitcoin's $f_\mathrm{Nak}$.
We note that the model in this paper is different than that of~\cite{gopalan2020stability} and so the one-endedness property for $f_\mathrm{Nak}$ needs to be re-established.

We note that one-endedness is only a sufficient condition for the existence of infinitely many confirmed blocks; the necessary and sufficient condition in Corollary~\ref{cor:confirmed} is expressed in terms of one-endedness of a spanning tree and is cumbersome to work with.
Generally, a lack of one-endedness should be thought of as having only finitely many confirmed vertices as $t \to \infty$, which defeats the purpose of the blockchain data structure.

\subsubsection{Implications of One-Endedness in Blockchain Applications}
\label{sssec:one-ended-implications}
In this section only, blocks are not atomic and we discuss their internal contents.
In a blockchain, having a path from the root block to another block indicates that along that path, there is no conflicting information contained in any blocks.
However, two different blocks $b_1$ and $b_2$ could store conflicting information.
In this case, the protocol should confirm at most one of these.

The one-endedness property implies that there are no infinite paths with conflicting information.
In fact, the one-endedness property also implies that for any two infinite paths $p_1$ and $p_2$, there are no blocks in $p_1$ which conflict with any block in $p_2$.
In other words, all such conflicts are resolved within a finite time.
We note that the finiteness guarantee of such conflict resolution occurs despite the fact that determination of one-endedness requires the $t\to\infty$ limit.
When designing a new blockchain architecture for new applications such as supply chains, it is therefore a crucial step to determine that the limiting blockchain is indeed one-ended.

The one-endedness property is a more general property which is not a generalization of any particular behavior of $f_\mathrm{Nak}$.
Using our model, the benefit thereof is a unified framework for establishing successful consensus for a variety of blockchain architectures.

As mentioned before, the one-endedness property can be approximated with heuristics for the Bitcoin blockchain since it operates in the inefficient regime where the arrival rate is almost zero.
However, even for $f_\mathrm{Nak}$, it is difficult to apply those heuristics in more standard operating regimes where the arrival rate is larger.
As mentioned above, it is also difficult to generalize these heuristics to other attachment rules.

\subsection{Overview of Main Results}

Recall that by \emph{asynchronous}, we refer to the setting where $\pr(\xi_1 = 1) < 1$.
We use the term \emph{asynchronous limit} to refer to the limit $G_\infty(f) := \lim_{t \to \infty}G_t(f)$ in the asynchronous setting with attachment rule $f$.

Our main results concern the development of techniques to determine the one-endedness of an asynchronous limit.
We focus on Bitcoin's $f_{Nak}$ and the set of functions $f_1, f_2, \ldots, f_\infty$ which are related to the Iota blockchain. Specifically, the Iota blockchain is widely modeled using $f_2$~(\cite{king2021fluid, feng2023fluid}).
For finite $k$, the function $f_k$ is such that the new block attaches to a set of $k$ leaves chosen uniformly at random without replacement (and attaches to all leaves if there are less than $k$ of them).
The function $f_\infty$ attaches the new block to all leaves.

Our main results can be summarized as follows. For careful statements of all results, we refer the reader to Section~\ref{ssec:main}.

The first two results are about the Bitcoin architecture $f_\mathrm{Nak}$.
\begin{theorem}
    The asynchronous limit $G_\infty(f_\mathrm{Nak})$ exists a.s. and is one-ended a.s..
\end{theorem}
\begin{theorem}
    The fraction of confirmed blocks in $G_\infty(f_\mathrm{Nak})$  is given by $\frac{1}{\E \chi}$, where $\chi$ is a random variable with law $\pr(\chi \geq k) = \prod_{i=1}^{k}\pr(\xi_1 \geq i)$.
\end{theorem}

The remaining results are for the Iota architecture $f_2$ and some other related architectures.
\begin{theorem}
    The asynchronous limit  $G_\infty(f_1)$ exists a.s. and is not one-ended a.s..
\end{theorem}
We note that, intuitively, the function $f_1$ is a poor candidate for a blockchain architecture because due to network delay, it may result in two infinite paths which only intersect at finitely many vertices, which is what is described in this result.
Nevertheless, the architecture $f_1$ is a useful tool for proving the following:
\begin{theorem}
    The asynchronous limits  $G_\infty(f_k)$, $2 \leq k \leq \infty$, exist a.s. and are one-ended a.s..
\end{theorem}
\begin{theorem}
    The following distributional convergence holds on an appropriate space:
    \[
		\begin{tikzcd}
			G_t(f_k) \arrow{r}{k\to\infty} \arrow{d}{t\to\infty} &G_t(f_\infty) \arrow{d}{t\to\infty} \\
			G_\infty(f_k)\arrow{r}{k\to\infty} &G_\infty(f_\infty)
		\end{tikzcd}
	\]
\end{theorem}

\subsection{Related Work}
We discuss related work in two contexts: (1) delay modeling in blockchains, and (2) properties and dynamics of the blockchain data structure.

There is a large literature, primarily concerned with analyses of Bitcoin, wherein a uniform bound on the network propagation delay of different blocks is assumed (see, \textit{e.g.},~\cite{dembo2020everything} and the references therein).
The first challenge with using this type of model is that it does not assume any particular structure of the network delay, other than the uniform bound.
This model structure is therefore unsuitable for the study of the dynamics of the blockchain data structure itself, outside of limited situations such as $f_\mathrm{Nak}$.
The second is that, as implemented, the node-to-node propagation delay is unbounded in the implementation of Bitcoin; and furthermore, the propagation dynamics depend on the peer-to-peer network size~(\cite{gopalan2020stability, gopalan2021data}). 
Thus, no uniform upper bound can be reasonably attained.
Our model handles unbounded delays, and can be extended with time-varying delay measures to handle dynamic network sizes and structures.

The works of ~\cite{gopalan2020stability} and~\cite{gopalan2021data} model the delays incurred by the peer-to-peer network as they are implemented in Bitcoin.
However, the state-spaces involved in analyzing those models with more general attachment functions than the ones considered in~\cite{gopalan2020stability} are untenable for analysis, and thus a new model is required to isolate the dynamics of the blockchain data structure from those of the peer-to-peer network.
Nevertheless, the results in this paper can be extended to the models considered in~\cite{gopalan2020stability} and~\cite{gopalan2021data} and we discuss the appropriate steps in Section~\ref{sec:disc}.

Some recent works have focused on the analysis of the Iota blockchain protocol and more specifically, the $f_2$ attachment rule.
\cite{king2021fluid} studies the dynamics of the number of leaves in the Iota blockchain.
However, this analysis does not provide detailed information about the blockchain data structure itself and thus does not inform how consensus is achieved in the Iota blockchain.

The one-endedness condition for blockchains was introduced in~\cite{gopalan2020stability}.
While it has been identified as an important topic for blockchain design (\cite{king2022some}), it is a difficult condition to work with.
In this article, we develop some techniques to determine the one-endedness of a limiting blockchain.

\subsection{Organization of the Paper}\label{ssec:intro-organization}

The remainder of this paper is structured as follows.
In Section~\ref{sec:defs-main}, we state our main results and the requisite definitions which we use in this paper.
We also describe the notation in this paper.
In Section~\ref{sec:disc}, we discuss our results and some directions for future work.
Proofs and other technical results are in the appendix.

\section{Definitions and Main Results}\label{sec:defs-main}

For the rest of this paper, the term~\emph{graph} always refers to a directed acyclic graph (DAG).

\subsection{Assumptions}
We use $\xi$ to refer random variables distributed identically to $\xi_1$ for clarity of presentation.
We will assume the following throughout the rest of the article:
\begin{itemize}
	\item $\pr(\xi\ge 1)=1$ and $\E \xi^{1+\gd} < \infty$ for some $\gd>0$,
\end{itemize}

\subsection{Notation}\label{ssec:defs-notation}
For the rest of the article, we will follow the notations enumerated below for easy reference.

\begin{itemize}
	\item For any real numbers $x, y$, we denote: $x \wedge y = \min(x, y);\
		 x \vee y = \max(x, y)$ and $x_+ = x \vee 0;\
		 x_- = (-x) \vee 0$.
	\item For a graph $G = (V, E)$, we use the notation $u \to v$ if there is a directed path from the vertex $u$ to the vertex $v$ in $G$.
	 It is clear from the definition of the asynchronous composition model that for any vertex $u$, $u \not \to u$.

	\item $\cbs (\obs)$ denote the set of all rooted, connected DAGs with finitely (infinitely) many vertices.
    All graphs in $\obs$ are locally finite.
	\item $\mvxi := (\xi_t)_{t\ge 1}$ is the i.i.d. driving sequence of delays.
	 We use the notation $\mvxi_i^j := (\xi_i, \xi_{i+1}, \ldots$, $\xi_j)$ for $1 \le i \le j \le \infty$.
	\item $\cF_t := \sigma(\xi_1, \ldots, \xi_t, \theta_1, \ldots, \theta_t)$ is the $\sigma$-algebra generated by the trajectories up to time $t$.
    Here, $(\theta_t)_t$ are i.i.d. $\mathrm{Unif}(0, 1)$ random variables used to set the random behavior of an attachment function $f$ - \textit{e.g.} to break ties in $f_\mathrm{Nak}$.
	\item If needed, we will use the notation $G_t(f)$ instead of $G_t$ to emphasize that the asynchronous composition is with respect to the function $f$.

	\item We denote by $(\gt_k)_{k\ge 1}$ the sequence of regeneration times as given in Definition~\ref{def:regeneration-time}, with $\gt_1\ge 0$ being the first regeneration time after time $0$.
	\item We denote by $\gc_i := \gt_{i+1} - \gt_{i}$ and $\tilde{\gc}_i := \tilde{\gt}_{i+1} - \tilde{\gt}_{i}$ for $i\ge 1$.

	\item We will use the calligraphic letter $\cA_t$ to denote a set at time $t$, and the corresponding roman letter $A_t$ to denote the cardinality of that set.
	 We will use the corresponding notation $\hat{\cA}_k=\cA_{\gt_{k}}$ to denote the same set at the $k$-th regeneration time, along with the corresponding notation $\hat{A}_k$. We also use the notation $\tilde{\cA}_k = \cA_{\tilde{\gt}_k}$ to denote the same set at the first instant of the $k$-th regeneration interval, along with the corresponding notation $\hat{\tilde{A}}_k$.

	\item We introduce the following:
	 \begin{itemize}
		 \item $\cL(G)$ denotes the set of leaves in the graph $G$, and $L(G)$ its size.
		 \item $\cI_{t, s} := \cL(G_s) \cap \cL(G_t)^c \text{ for }t \ge s,$ is the set of leaves at time $s$ which are not leaves at time $t$.
		 $I_{t, s}$ is the size of $\cI_{t, s}$.
	 \end{itemize}

	\item We will use the shorthand $\cL_{t},L_{t}$ for $\cL(G_{t}), L(G_{t})$, respectively. Similarly, we will use $\hat{\cL}_{k}, \hat{L}_{k}$, $\tilde{\cL}_k$, $\tilde{L}_k$ for $\cL(G_{\gt_k}), L(G_{\gt_k}), \cL(G_{\tilde\gt_k}), L(G_{\tilde\gt_k})$, respectively.
\end{itemize}

\subsection{Infinite Graphs: Definitions, One-Endedness, and Blockchains}\label{ssec:defs-infinite-graphs}

As noted in the introduction, we take an asymptotic approach to block confirmation in this paper.
In this section we make precise the sense in which temporal limits exist.
We also define the notion of ends of an infinite graph.
Since we use an asymptotic definition of block confirmation, we note that the foremost criterion of interest is to guarantee the existence of infinitely many confirmed blocks.
To do this, we explicitly characterize the number of ends in terms of the existence of infinitely many confirmed blocks.

A graph $G = (V, E)$ is infinite if $V$ is infinite. An infinite graph $G$ is~\emph{locally finite} if all vertices $v \in V$ have finite degree.

\subsubsection{The Spaces $\cbs$ and $\obs$}
In this section, we introduce the space $\cbs$ as well as the space $\obs$, in which all graph limits in this paper exist.
Specifically, let $\cbs$ be the set of all rooted, finite, connected DAGs, and let $\obs$ be the set of all rooted, locally finite, connected DAGs.
We endow $\cbs$ with the metric $d_*$, defined as follows.
\begin{definition}[{\cite[Chapter 2]{aldous2007processes}}]
	The function
	$d_*(G_1, G_2) := \frac1{1+s}$,
	where $s$ is the supremum of all integers such that the $s$--balls w.r.t.~the hop distance centered at the roots of $G_1$ and $G_2$ agree, is a metric on $\cbs$.
\end{definition}
It is easily checked (see~\cite{aldous2007processes}) that $\obs$ is the closure of $\cbs$ under $d_*$ and that $(\obs, d_*)$ is a complete metric space.

\subsubsection{Ends in Infinite Graphs}
\label{sssec:ends}
In this section, we introduce the notion of \emph{ends} of an infinite graph.
We discuss the interpretation of these definitions separately in Section~\ref{ssec:defs-infinite-blockchain}.

A \emph{ray} is a semi-infinite directed path in an infinite graph $G\in \obs$.

\begin{definition}[{See~\cite{halin1964unendliche}}]\label{def:eqv}
	Two infinite rays $r_1$ and $r_2$ in $G$ are~\emph{equivalent} if there exists a third infinite ray $r_3$ such that $|r_1 \cap r_3| = |r_2 \cap r_3| = \infty$, where the intersection is taken over vertices.
\end{definition}

It is readily checked that being equivalent defines an equivalence relation on the set of rays in a fixed infinite graph $G \in \obs$.

\begin{definition}[{See~\cite{halin1964unendliche}}]
	The graph $G$ is~\emph{$n$-ended} if the equivalence relation in Definition~\ref{def:eqv} separates infinite rays of $G$ into $n$ equivalence classes; each class is called an~\emph{end}.
	If there is only a single equivalence class, $G$ is~\emph{one-ended}.
    If there are infinitely many equivalence classes, $G$ has infinitely many ends.
	If there are no infinite rays, $G$ has~\emph{$0$ ends}.
\end{definition}

We use the convention that any finite graph has $0$ ends. Moreover, due to K\"onig's Lemma, any locally finite infinite graph has at least one end.
From this definition, it is clear that the number of ends in an infinite graph cannot be inferred from the properties of any finite subgraph.


\subsubsection{Infinite Graphs and Blockchain}
\label{ssec:defs-infinite-blockchain}
In this section, we provide interpretations of the definitions in Section~\ref{sssec:ends} and an explicit characterization of one-endedness in terms of confirmed blocks.

First, note that (asymptotically), successful operation of a blockchain protocol requires that for any arriving vertex $v$, all rays from the root to $v$ are equivalent; this implies that there are no conflicting information contained in any block along any path from the root to $v$.
Thus, equivalency of these rays implies consistency of the information in the blockchain.
With this interpretation in mind, one should also require that \emph{all} infinite rays are equivalent; so that (asymptotically) all rays to any arriving vertex are conflict-free (we refer the reader to the discussion in Section~\ref{sssec:one-ended-implications}).

\begin{definition}
    A vertex $v$ in the (synchronous or asynchronous) limit of the function $f$ is~\emph{confirmed} if $v \to w$ for all but finitely many $w > v$.
\end{definition}

We state a lemma from~\cite{gopalan2020stability} which identifies crucial properties of limiting blockchain graphs.
In the interest of self-containedness, we include proof of this lemma.

\begin{lemma}[{\cite[Lemmas 3.4 and 3.5]{gopalan2020stability}}]
	If a graph $G \in \obs$ is one-ended, then it has infinitely many confirmed vertices.
	Conversely, if $G$ has infinitely many confirmed vertices, then there is a one-ended subgraph of $G$ which contains all of the confirmed vertices.
	\label{lem:blockchain-one-ended}
\end{lemma}
\begin{proof}{Proof}
	Suppose that $G$ is one-ended.
	Fix any ray $r_1$; we will show that each vertex contained in $r_1$ is confirmed.
	For any other infinite ray $r_2$, we have a ray $r_3$ which intersects both $r_1$ and $r_2$ infinitely often.
	This implies that for any vertex $v$ in $r_1$, all but finitely many vertices in $r_2$ have a path to $v$.
	This part of the result then follows since $G$ is locally finite.

	Next, suppose that $G$ has infinitely many confirmed vertices and denote by $\widehat{G}$ the subgraph of the confirmed vertices.
	Consider two infinite rays $r_1, r_2$ in $\widehat{G}$.
    We will show that $r_1$ and $r_2$ are equivalent.
    Fix a vertex $v_1 \in r_1$. The result follows if there is a path from $v_1$ to all but finitely many vertices in $r_2$; this latter condition is immediate from the definition of a confirmed vertex.
\end{proof}

We note that a \emph{spanning tree} for a graph $G$ is a subgraph $G' = (V, E')$, where $E' \subseteq E$, the root in $G'$ is the same as the root in $G$, and each (non-root) vertex in $G'$ has a unique path to the root.
We add the following easy corollary, which is a new result:

\begin{corollary}
	A graph $G \in \obs$ has infinitely many confirmed vertices iff it has a one-ended spanning tree.
    \label{cor:confirmed}
\end{corollary}

In practice, it is far easier to check the one-endedness of a graph than to establish the existence of a one-ended spanning tree. So we do not use the corollary even if it expresses a tighter condition for the existence of infinitely many confirmed vertices.
It follows from Lemma~\ref{lem:blockchain-one-ended} that a critical question related to the design of blockchain architectures is the determination of which one-ended functions have one-ended asynchronous limits.

\subsubsection{Some Technical Lemmas}
\label{subsec:technical-ends}
The following technical lemmas are helpful in our analysis, and we put them here to simplify the presentation later in the paper.

\begin{lemma}
	\label{lem:tree-infty}
	Let $(T_t)_t$ be a sequence of finite trees with $T_t \subseteq T_{t+1}$ for all $t \ge 0$.
	Suppose the number of leaves is non-decreasing in $t$ and diverges to infinity, and that any leaf in $T_t$ is such that for some $s > t$, that leaf is not a leaf in $T_s$.
	If $T := \lim_{t \to \infty}T_t$ exists in $\obs$, then $T$ has infinitely many ends.
\end{lemma}
\begin{proof}{Proof}
	Fix any graph $T_t$.
	Any leaf in $T_t$ is part of an infinite path in $T$.
	Thus, if there are $k$ leaves in $T_t$, then $T$ has at least $k$ ends.
	The result follows since the number of leaves in $T_t$ tends to infinity.
\end{proof}

\begin{lemma}
	\label{lem:tree-one}
	Let $T \in \obs$ be an infinite tree.
	$T$ is one-ended iff it has infinitely many confirmed vertices.
\end{lemma}
\begin{proof}{Proof}
	If $T$ is one-ended, then it has infinitely many confirmed vertices by Lemma~\ref{lem:blockchain-one-ended}.
	Suppose $T$ has infinitely many confirmed vertices.
	Since $T$ is a tree, there exists an infinite path $p$ consisting of confirmed vertices.
	However, since $T$ is a tree, all infinite paths must intersect $p$ infinitely often.
\end{proof}

From the definition of one-endedness, it follows easily that for an infinite graph $G$, $G$ is one-ended iff any two rays are equivalent.

\begin{lemma}\label{lem:1end}
	Let $G \in \obs$ be an infinite graph.
	Suppose that there is an infinite sequence of vertices $(v_i)_i$ such that any infinite path passes through $v_i$ for all $i$.
	Then $G$ is one-ended.
\end{lemma}
\begin{proof}{Proof}
	In this case, all rays are clearly equivalent. The result follows from the definition.
\end{proof}{Proof}
In Lemma~\ref{lem:1end}, the vertices in the sequence $(v_i)_i$ can be thought of as~\emph{anchor vertices}.

\subsection{Model}
Our model evolves in discrete time.
Let $f:\cbs \times [0, 1] \to \cbs$ be an \emph{attachment rule} such that for any finite graph $G \in \cbs$, $f(G)$ contains one new vertex and at least one edge from an existing vertex to the new vertex.
The attachment rule $f$ is also such that the root of $f(G)$ is the same as the root of $G$ for any $G \in \cbs$.
Let $(\xi_i)_{i \in \bN}$ be an i.i.d. sequence of $\bN$-valued random variables such that $\pr(\xi_1 = 1) > 0$ and such that for some $\delta > 0$, we have $\E\xi_1^{1 + \delta} < \infty$.

Let $G_0$ be a graph containing only a root vertex and no edges.
We consider the dynamics:
$$G_t = G_{t-1} \cup f(G_{(t - \xi_t)_+}, \theta_t).$$
The argument $\theta_t \sim \mathrm{Unif}(0, 1)$ is used for technical correctness so that the attachment rule $f$ can be a ``random function;'' in the sequel we will suppress the argument $\theta_t$.

The interpretation of the model is as follows.
If the node creating the $t$-th block experiences a delay of $(t - \xi_t)_+$, then there are $(t - \xi_t)_+$ many blocks which have yet to be communicated to that node.
For simplicity, we assume that these are the $(t - \xi_t)_+$ most recent blocks.
However, those blocks are not actually missing from the blockchain; instead they are only unavailable at the node creating the $t$-th block.
Hence, the function $f$ should be applied to $G_{(t - \xi_t)_+}$, and the result should be taken in union with $G_{t-1}$.

Recall that the term \emph{synchronous} refers to the setting where $\pr(\xi = 1) = 1$ and that the term \emph{asynchronous} refers to the setting where $\pr(\xi = 1) < 1$.

\begin{definition}
	The~\emph{synchronous limit} is given by
	$f^\infty(G_0; \mvtheta,\mvone)
		:= \lim_{n\to\infty}f^n(G_0; \mvone)$,
	where the limit is with respect to the metric $d_{*}$.
\end{definition}

For all functions considered in this paper, the existence of the synchronous limit is immediate, and we omit proofs for brevity.

\begin{definition}
	The~\emph{asynchronous limit} is given by
	$G_\infty(f) := \lim_{t\to\infty}G_t(f) = \lim_{t \to \infty}f^t(G_0, \mvxi_1^t)$,
	where the limit is w.r.t.~the $d_{*}$ metric.
\end{definition}

\subsection{Structural Properties of the Delay Process}
For the rest of this paper, we will assume that $\pr(\xi = 1) > 0$ unless otherwise stated.
Most of our results can readily generalized to the situation where $\pr(\xi = 1) = 0$ and we include those results in the appendix.

The following definitions provide an important structural framework for our analysis.

\begin{definition}
	An integer $t\ge 0$ is a~\emph{regeneration time} for the delay sequence $(\xi_{s})_{s\ge 0}$ if $\xi_t = 1$ and $\xi_{t+s} \le s$ for all $s\ge 1$.
	\label{def:regeneration-time}
\end{definition}
In other words, if $t$ is a regeneration time and $s > t$, then $x - \xi_s \geq t$.
Note that $t$ is a regeneration time iff $1\le \xi_{t+s}\le s\vee 1$ for all $s\ge 0$.
It can be easily checked that regeneration times have the Markov property.
In Appendix~\ref{sec:preliminary-results}, we show that there are infinitely many regeneration times and that the inter-regeneration duration is finite.
This allows us to analyze our process by considering functionals of the blockchain data structure at the regeneration times; the induced behavior results in Markov processes on $\bN$.

\subsection{Main Results}\label{ssec:main}
We introduce the following functions which are the main focus of our analysis:
\begin{itemize}
	\item $f_{\Nak}$ is the Nakamoto function, where a vertex is chosen uniformly from those at the maximum hop distance from the root.
	\item In $f_1(G)$ a single leaf is chosen uniformly at random from $G$.
	\item For $k \ge 2$, $f_k(G)$ chooses a uniformly selected set of $k$ leaves from $G$ if possible; otherwise all leaves in $G$ are chosen.
	\item In $f_\infty(G)$ all leaves in $G$ are chosen.

	\item We denote by $f$ any random mixture of $(f_i)_{i\ge 1}$ such that $\pr(f = f_1) < 1$.
\end{itemize}

It is clear that all of $f_{\Nak}, (f_k)_{k\ge 1}, f, f_\infty$ have one-ended synchronous limits. 
Our main results are as follows.

\begin{lemma}
    Synchronous and asynchronous limits exist a.s. in $\obs$ for all of the aforementioned attachment rules.
    \label{lem:limit-existence}
\end{lemma}

\subsubsection{Results for Bitcoin's $f_\mathrm{Nak}$}
We begin with an analysis of the Nakamoto construction $f_{\Nak}$ which is used in Bitcoin. It is easy to check that $G_t(f_{\Nak})$ is a tree for all $t\ge 0$.
\begin{theorem}\label{thm:nakamoto-ends}
	The asynchronous limit of $f_{\Nak}$ is one-ended, almost surely.
\end{theorem}




It follows from the definition of $f_\mathrm{Nak}$ that the  asynchronous recursion determines the length of the longest path from the root to any leaf, or the height of the tree, at time $t$ for $f_{\Nak}$: 
\begin{align}
	\begin{split}\label{eq:nak-recursion}
		X_{t} & = X_{t-1} \vee (1 + X_{(t - \xi_t)_+}),\ t \ge 1, \\
		X_0 & = 0.
	\end{split}
\end{align}

\begin{theorem}
	\label{thm:nakamoto-conv}
	Let $\chi$ be an integer-valued random variable with $\pr(\chi \ge k) = \prod_{i=1}^{k}\pr(\xi \ge i)$, for $k \ge 1$. We have:
	$\frac{X_t}{t} \xrightarrow{a.s.} \gl :=\frac1{\E(\chi)}$ and $\frac1{\sqrt{t}}\cdot (X_t - \gl t) \xrightarrow{(d)} \textsc{N}(0, \gl^3\var(\chi))$
	as $t\to\infty$.
	Furthermore, $\frac1{n}X_{\lfloor nt \rfloor} \to \gl t$ converges uniformly a.s.~on the compact subsets of $\dR_+$ as $n\to\infty$.
	Define
	$Z_n(t) := n^{-1/2}\cdot (X_{\lfloor nt \rfloor} - \gl n t), t\ge 0$.
	Then $Z_n(\cdot) \xrightarrow{w} Z(\cdot)$, which is a zero-drift Brownian motion with variance parameter $\gl^3 \var(\chi)$.
\end{theorem}

\begin{remark}
    It is easy to biject the instants when $X_t$ increases by exactly one with the confirmed blocks in $G_\infty(f_{\text{Nak}})$. Thus, the recursion~(\ref{eq:nak-recursion}) also characterizes the fraction of blocks which are confirmed in the asynchronous limit.
\end{remark}

\begin{remark}
	Note that, in Theorem~\ref{thm:nakamoto-conv} the random variable $\chi$ has moments of all order as $\pr(\chi\ge k) \le \pr(\xi\ge k)^{k} \le (\E\xi/k)^{k}$ for all $k\ge 1$.
\end{remark}

\begin{remark}
	When $\xi\sim \text{Geometric}(p)$ with $\pr(\xi=k)=p(1-p)^{k-1}, k\ge 1$, we have $$\pr(\chi\ge k)=\prod_{i=1}^{k}(1-p)^{i-1} = (1-p)^{k(k-1)/2},\ k\ge 1.$$ Thus, in this particular example, the asymptotic growth rate of the longest chain in Theorem~\ref{thm:nakamoto-conv} is given by $\gl=1/\sum_{k=1}^{\infty}(1-p)^{k(k-1)/2}$. This is related to the Jacobi Theta Functions. It is an interesting question on how to estimate $p$ based on sample observations of the blockchain height.
\end{remark}


\subsubsection{Results for Iota's $(f_k)$, $k \geq 1$}
Next, we present the results for $f_{k},k\ge 1$ and their mixtures.
\begin{theorem}	\label{thm:f_1-growth}
	The asynchronous limit $G_\infty(f_1)$ has infinitely many ends, almost surely.
	Furthermore, the expected number of leaves in $G_t(f_1)$ is $\Theta_{\pr}(t^{1/2})$.
\end{theorem}

\begin{remark}
	One can guess from the results of the above Theorem~\ref{thm:f_1-growth} that $t^{-1/2}\cdot L_t$ converges in distribution to some non-trivial limit as $t\to\infty$; however, we do not pursue this result here.
\end{remark}



For $f_{k}$ with $k\ge 2$ or $f$ being a mixture of $f_{k}$'s with $\pr(f=f_1)<1$, we have the same endedness behavior for the synchronous and the asynchronous model as stated below.

\begin{theorem}	\label{thm:f_2-recurrent}
	The asynchronous limits $G_\infty(f_k)$ and $G_\infty(f)$ are one-ended, almost surely, for $k \ge 2$.
\end{theorem}
\begin{remark}
	Limiting statements (similar to Theorem~\ref{thm:nakamoto-conv}) about linear growth and distributional and process limits of the height process can be proved for the construction functions in Theorem~\ref{thm:f_2-recurrent}; we omit the statements and proofs for brevity.
\end{remark}





\subsection{Constructing New Architectures from Old Ones}
When the attachment function is graph dependent, one can prove a phase transition.

\begin{theorem}	\label{thm:state-varying}
	Define
	\begin{align*}
		\hat{f}_{k, l} :=
		\begin{cases} f_{k} & \text{w.p. } \frac{\ga}{\sqrt{l}} \\
 f_1 & \text{otherwise}.
		\end{cases}
	\end{align*}
	For fixed $k\ge 2$, define the function $g: \cB_* \times [0, 1] \to \cB_*$ such that $g(\cdot) = \hat{f}_{k, l}(\cdot)$ whenever the argument has $l$ leaves.
	There exist constants $0 < c_1 < c_2 < \infty$ such that the asynchronous limit is $g$ is one-ended if $\ga > c_2$ and has infinitely many ends if $\ga < c_1$, almost surely.
\end{theorem}

Finally, we will prove the following limiting commutative diagram behavior.

\begin{theorem}	\label{thm:commuting-diagram}
	The following diagram commutes
	\[
		\begin{tikzcd}
			G_t(f_k) \arrow{r}{k\to\infty} \arrow{d}{t\to\infty} &G_t(f_\infty) \arrow{d}{t\to\infty} \\
			G_\infty(f_k)\arrow{r}{k\to\infty} &G_\infty(f_\infty)
		\end{tikzcd}
	\]
	where the convergence holds in the sense of distributional convergence in the space $(\obs, d_*)$.
	Moreover, with coupled delays, the convergences are almost sure if $\E\xi^{1 + \delta} < \infty$ for some $\delta > 0$.
\end{theorem}


\subsection{Model Generalization}
\label{ssec:model-gen}

We note that our results hold in a special case of the following generalization of the model.
This generalization is based on combining the ``forward'' delays in~\cite{baccelli2019renewal} and~\cite{king2021fluid}, with the ``backward'' delays in our model.
In an application context, the forward delays represent the computation time required to create a new block and the backward delays represent the time required to access data.

Denote by $(\nu_t)_t$ a sequence of i.i.d. $\dN$-valued random variables with irreducible support ($\mathrm{gcd}(\{n \in \dN: \pr(\nu_1 = n) > 0\}) = 1$).
We assume that $\E\nu_1 < \infty$ and define
$$ G_t = G_{t-1} \cup f(\overline{G}_{(t - \xi_t)_+}),$$
where $\overline{G}_{(t - \xi_t)_+}$ is a subgraph of $G_{(t - \xi_t)_+}$ which consists of those vertices $s$ such that $s + \nu_s \le (t - \xi_t)_+$.
This corresponds to the vertex $t$ taking $\nu_t$ units of time to be created, and then to begin propagating, which is more realistic in the blockchain application context.

The ``forward'' delay process connecting $t$ to $t + \nu_t$ for $t\ge 1$ has renewals which are stopping times.
When $\pr(\xi = 1) > 0$, all of our results hold as-is since the intersection of independent renewal processes is again a renewal process with well understood gap distribution (see~\cite{AB16}), and since the regeneration times posses the required Markov property.
We specifically point out the difference in the requirements on $\nu_1$ and $\xi$ for this setting: for $\nu_1$, we require irreducibility of the support, but for $\xi$, we require only that $\pr(\xi = 1) > 0$.
\section{Discussion and Further Questions}\label{sec:disc}
This article introduces a new model as a tool for the asymptotic analysis of blockchain architectures.
We show its efficacy for studying the end structure of limiting blockchain graphs and by studying the Nakamoto construction from the Bitcoin protocol and the construction $f_2$ from the Iota protocol.
Using the related idea of an asynchronous recursion, we also explicitly characterize the growth rate of the longest path to the root under the Nakamoto function and the fraction of confirmed blocks.
This rate is the key parameter used in the security analyses of $f_\mathrm{Nak}$ but has not been previously characterized.

A large variety of interesting directions for future work can be considered.

\begin{enumerate}\setlength{\itemsep}{1ex}
	\item \textit{Fix an attachment function $f$, and let $\Xi$ be the space of probability distributions on $\dN$ such that for all $\xi \in \Xi$, $0 < \pr(\xi = 1) < 1$ and $\E \xi^{1 + \delta} < \infty$ for some $\delta > 0$.
	Is the number of ends in the asynchronous limit invariant to $\Xi$? If not, can we determine the class of attachment functions for which one-endedness holds and is invariant to the choice of distribution from $\Xi$?}
	This question is of significant practical importance to this question in the context of blockchain systems:
    for all functions with asynchronous limits for all choices of $\xi \in \Xi$, even if the underlying network changes (\textit{e.g.} network topology or network size), the one-ended limit should still hold.

	\item \textit{Fix $\xi \in \Xi$ as above. For which attachment rules $f$ does a one-ended synchronous limit correspond to a one-ended asynchronous limit?}
	 This question can be interpreted as, ``given any fixed network behavior, can any one-ended construction function be used in a blockchain system?''
	 We conjecture that this statement is true, at least when $\pr(\xi_1 = 1) > 0$.
	 Indeed, an intuition for this may be as follows.
	 Since the function is one-ended, the synchronous limit is one-ended for any finite graph.
	 Thus, in asynchronous operation, if the delay process satisfies $\xi_t = 1$ for infinitely many sufficiently long disjoint intervals, one may be able to conclude that the asynchronous limit is also one-ended.
	 This intuition is a critical idea in many of the proofs in this paper.

    \item \textit{Can the i.i.d. assumption on $(\xi_i)_i$ be relaxed?}
    In particular, if $(\xi_i)_i$ are stationary ergodic Markov, then the one-endedness results in our model also hold in the model considered by~\cite{gopalan2020stability} and~\cite{gopalan2021data}; recall that this model is accurate to how the Bitcoin network is implemented.

    \item \textit{What is the time to consensus? }
    Block confirmation is asymptotic in the sense that if a block $b$ is confirmed, then \emph{eventually} it will have a path to all arriving blocks.
    The time to consensus is the distribution of this eventual event.
    In particular, the time to consensus may be viewed as a blockchain's analog to a three-business-day credit card processing time.
    Of course, an appropriate notion of stationarity for the blockchain data structure would facilitate computation of the distribution of time to consensus for a general architecture.

    \item \textit{Stationarity of and Security in the Blockchain Data Structure: }
    In addition to providing a suitable setting to compute the time to consensus, developing a notion of stationarity for the blockchain data structure also informs eventual security analyses for general blockchain architectures.
    We note that security is well-understood for $f_\mathrm{Nak}$~(\cite{dembo2020everything}), but is otherwise an open question.
    Stationarity of the blockchain data structure provides a notion that, imprecisely, all confirmed blocks contribute in the same way to the confirmation process.
    One implication of such stationarity is that the blockchain height grows linearly, and thus (eventually) confirmed blocks are added at a constant rate.
    One-endedness by itself is not sufficient to guarantee this.

    \item \textit{Generalized Blockchain Height Recursion for $f_\mathrm{Nak}$: }
    Consider the process
    \begin{align*}
    \begin{split}
		X_{t} & = X_{t-1} \vee (\eta_t + X_{(t - \xi_t)_+}),\ t \ge 1, \\
		X_0 & = 0.
	\end{split}
    \end{align*}
    When $\pr(\eta_t = 1) = 1$, we recover the same process as in Recursion~(\ref{eq:nak-recursion}).
    When $\pr(\eta_t = 1) < 1$, our analysis is difficult to extend, especially in the case where $\eta$ has non-compact support.
    This is useful when the discretization of time is such that the participant(s) creating a block can create several blocks in one time step, leading to the necessity of this generalization.
    This particular generalization may be of key importance especially in the situation where blocks arrive more quickly, as is the motivation for this paper.
    
\end{enumerate}

\appendix

\section{Regenerative Behavior}\label{sec:preliminary-results}

Recall our standing assumption that $\E\xi<\infty$. Here we will analyze structure of the time-delay graph depending on whether $\pr(\xi=1)>0$ or not. In the first case, we will prove existence of infinitely many~\emph{pivotal points} or ``regeneration points'' giving linear structure for the time-delay graph. In the second case, there is almost surely no regeneration points. In fact, depending on the g.c.d.~of $\supp(\xi)$ the time-delay graph can have a periodic structure. However, we will show that there exists infinitely many ``regeneration intervals'', disconnecting the future from the past. The regenerative structure is one crucial ingredient for the subsequent analysis.

\subsection{Regeneration Times: $\pr(\xi=1)>0$}

We define
$
	\cE_{t}:= \bigcap_{s\ge 0}\{\xi_{s+t} \le s\vee 1\}
$
as the event that $t$ is a regeneration time.
Recall that the delay random variables $(\xi_{s})_{s\ge 0}$ are i.i.d.~$\sim\xi$. Thus, we have for all $t\ge 0$
\begin{align*}
	\pr(\cE_{t}) =q:= \pr(\xi = 1)\prod_{s=1}^\infty \pr(\xi \le s).
\end{align*}
Note that, $q>0$ as $\pr(\xi = 1)>0$ and $\E\xi = \sum_{s=0}^\infty \pr(\xi> s) <\infty$. We will use $\ga:=\pr(\xi = 1)$.
We also define
\begin{align*}
	N_{n}:=\sum_{t=0}^{n}\1_{\cE_{t}}
\end{align*}
as the number of regeneration points in the time interval $\{1,2,\ldots,n\}$.
We can compute the mean and variance of $N_{n}$ easily.

\begin{lemma} \label{lem:weak-law}
	Assume that $\pr(\xi=1)>0$ and $\E\xi <\infty$. Then $\E N_n=nq$ for all $n\ge 1$ and ${n^{-2}} \var(N_n)\to 0$ as $n\to\infty$.
\end{lemma}

If we assume that $\E\xi^2<\infty$, then it follows from the proof that $n^{-1}\var(N_{n})$ converges to a constant as $n\to\infty$. But, we do not need this result for our analysis. As an immediate corollary of Lemma~\ref{lem:weak-law} we get the following result.
\begin{corollary} \label{cor:regeneration}
	There exist infinitely many regeneration times, almost surely.
\end{corollary}
\begin{proof}{Proof.}
	It is easy to see that $(N_{n})_{n\ge 1}$ is an increasing sequence of random variables converging a.s.~to some integer-valued random variable $N_{\infty}$, which can possibly take the value $\infty$.
	Using Lemma~\ref{lem:weak-law} and Chebyshev's inequality we have for any $\eps>0$
	\begin{align*}
		\pr\left(|{N_n}/n - q| > \eps\right) \le \var(N_{n})/\eps^2n^2\to 0 \text{ as } n\to\infty.
	\end{align*}
	Thus $N_{n}/n\to q$ in probability as $n\to\infty$. Since $q>0$, this proves that $N_{\infty}=\infty$ a.s.
\end{proof}

The following corollary follows from similar arguments to above; we omit the proof.
\begin{corollary}	\label{cor:regeneration-fixed-gaps}
	There exist infinitely many regeneration times $(T_i)_i$ such that $T_i + 1$ is also a regeneration time, almost surely.
\end{corollary}

Denote by
\begin{align*}
	0\le \gt_1<\gt_2<\cdots < \cdots
\end{align*}
an increasing enumeration of all the regeneration points in the interval $\{0,1,2,\ldots,\infty\}$.

We consider a more general delay process for all times $t \in \dZ$, from which the system dynamics at all times $t \in \dZ_+$ are uniquely determined.
Specifically, let $(\xi_i)_{i \in \dZ}$ be i.i.d.
We denote by
\begin{align*}
	J_t := (t - \xi_t, t) \subseteq \dZ
\end{align*}
for all $t \in \dZ$.
Denote by $A := \cup_{t \in \dZ} J_t$.
It follows that the random set $\dZ \setminus A := \{\ldots < \gt_{-1} < \gt_{0} < 0 \le \gt_1 < \ldots \}$ is such that
\[
	\gc_k := \gt_{k+1} - \gt_k
\]
are i.i.d.~for all $k \in \dZ$. For $k \ge 1$, the times $\gt_{k}$ are precisely the regeneration times from the time-delay process.

\begin{lemma}\label{lem:reg-prop}
	The random variables $(\gc_k)_{k\ge 1}$ are i.i.d.~with $\E\gc_1 =1/q$. Moreover, if $\E \xi^k<\infty$ for some $k\ge 2$, then $\E\gc_1^k <\infty$.
\end{lemma}
\begin{proof}{Proof.}
	The fact that $(\gc_k)_{k\ge 1}$ are i.i.d.~mainly follows from the fact that conditional on the event $\cE_{t}= \{t $ is a regeneration time$\}$, $(\xi_{t+s})_{s\ge 1}$ is distributed as independent $(\hat\xi_{s})_{s\ge 1}$ where $\hat\xi_{s}\equald (\xi\mid \xi\le s), s\ge 1$ and that the event $\cE_{t}$ depends only on the future, $(\xi_{s})_{s\ge t}$.

	For an event $A_{0,s}$ depending only on $\xi_{i}, i\in (0,s]$, we write $A_{t,s}$ when the random variables $\xi_{i}, i\in (0,s]$ are replaced by $\xi_{t+i}, i\in (0,s]$. Fix $0\le t_1<t_2<\ldots <t_{k}$ and events $A^{(i)}_{t_{i},t_{i+1}-t_{i}}, i\ge 1$. We also use $\hat{\pr}$ to denote $\pr(\cdot\mid \cE_{0})$, \ie\ the case when $\xi_{t},t\ge 1$ is replaced by independent $\hat{\xi}_{t}\equald (\xi\mid \xi\le t)$. Thus we have
	\begin{align*}
		 & \pr(\gt_1=t_1, \gt_{i+1}=t_{i}, A^{(i)}_{t_{i},t_{i+1}-t_{i}}, 1\le i< k) \\
		 & \quad = \pr(\gt_1=t_1)
		\cdot \hat{\pr}(\gt_{i}=t_{i+1}-t_1, A^{(i)}_{t_{i}-t_1,t_{i+1}-t_{i}}, 1\le i\le k-1).
	\end{align*}
	By induction, this equals
	\[
		\pr(\gt_1=t_1)\cdot \prod_{i=1}^{k-1}\hat{\pr}(\gt_1=t_{i+1}-t_i, A^{(i)}_{0,t_{i+1}-t_{i}}).
	\]
	This proves the i.i.d.~structure for $(\gc_{k})_{k\ge 1}$.

	From the independence result, it follows that at the regeneration times, the graph $G_{\gt_k}$ satisfies the Markov property.
	The fact that $\E\gc_1=1/{q}$ follows from the renewal theorem.
	Finally, the second part of this lemma follows from the following Propositions~\ref{prop:gamma-tau} and~\ref{prop:tau-xi}.
\end{proof}

\begin{proposition}
	$\E \gc_1^\gb < \infty$ iff $\E \gt_1^{\gb - 1} < \infty$, for any fixed $\gb\ge 1$.
	\label{prop:gamma-tau}
\end{proposition}

\begin{proposition}
	$\E \gt_1^{\gb - 1} < \infty$ if $\E \xi^\gb < \infty$, for any fixed $\gb\ge 1$.
	\label{prop:tau-xi}
\end{proposition}
Proofs of Propositions~\ref{prop:gamma-tau} and~\ref{prop:tau-xi} are given in Sections~\ref{sec:propgt} and~\ref{sec:proptx}, respectively.

\subsection{Regeneration Intervals: $\pr(\xi = 1) = 0$}
In this case we consider the more general setting where $\pr(\xi = 1) \ge 0$; for which it suffices to assume that $\pr(\xi = 1) = 0.$
This case can be interpreted as allowing the minimum delay to be greater than $1$.
Similar results hold for regeneration intervals as for regeneration times, which we state next; for brevity we omit the proofs.

Throughout, let $r := \min_{n \in \bN}\{n: \pr(\xi = n) > 0\}$.

Define
\begin{align*}
	\tilde{\cE_t} := \bigcap_{s \in [0,r)}\{\xi_{t+s} = r\}\bigcap_{s\ge r}\{\xi_{t+s} \le s\} = \bigcap_{s\ge 0}\{\xi_{t+s} \le s\vee r\}
\end{align*}
to be the event that the interval $[t, t + r)$ is a regeneration interval.
Since the $(\xi_i)_i$ are i.i.d. we have
\begin{align*}
	\pr(\tilde{\cE}_t) := \prod_{s \ge 0}\pr(\xi \le s\vee r).
\end{align*}
Similar to above, we define
\begin{align*}
	\tilde{q} := \prod_{s \ge 0}\pr(\xi \le s\vee r) > 0,
\end{align*}
since $\E \xi < \infty$. We define
\begin{align*}
	\tilde{N}_n := \sum_{t = 0}^{n}\ind_{\tilde{\cE}_t}
\end{align*}
as the number of regeneration windows in the first $n + 1$ segments.

\begin{lemma}
	$\E \tilde{N}_n=n\tilde{q}$ for all $n\ge 1$ and ${n^{-2}} \var(\tilde{N}_n)\to 0$ as $n\to\infty$.
	\label{lem:weak-law-r}
\end{lemma}

In what follows we refer to regeneration windows by the first time in those windows; this is without loss of generality by the construction of the segmented time.
\begin{corollary}
	There exist infinitely many regeneration windows, almost surely.
\end{corollary}
We denote by $(\tilde{\gt}_k)_{k\ge 1}$ the sequence of times such that $[\tilde{\gt}_k, \tilde{\gt}_k + r)$ are regeneration intervals.
\begin{corollary}
	There exists infinitely many times $(T_i)_i$ such that $[T_i, T_i + r)$ and $[T_i + r, T_i + 2r)$ are both regeneration intervals, almost surely.
\end{corollary}

\begin{lemma}
	\label{lem:reg-prop-r}
	The random variables $(\tilde{\gc}_i)_{i\ge 1}$ are i.i.d. with $\E\tilde{\gc}_1 = 1/{\tilde{q}}$.
	Moreover, if $\E \xi^k < \infty$ for some $k > 1$, then $\E \tilde{\gc}_1^k < \infty$.
\end{lemma}
Lemma~\ref{lem:reg-prop-r} follows from an identical argument to Lemma~\ref{lem:reg-prop}, which can be seen as follows.
Indeed, suppose that time is ``pre-chunked'' into intervals $[rk + a, r(k+1) + a)$, where $k \in \dZ$ and $a$ is fixed.
Indeed, on these chunks, the previous argument holds to identify regeneration windows; the result follows since $\mvxi$ is i.i.d. and $a$ is arbitrary.

\subsection{Proof of Lemma~\ref{lem:limit-existence}}
The proof of Lemma~\ref{lem:limit-existence} is now immediate as upon a regeneration time, the degree of any vertex which is not a leaf is henceforth fixed over time.

\section{Proofs for Regenerative Behavior}
\subsection{Proof of Lemmas~\ref{lem:weak-law} and~\ref{lem:weak-law-r}}\label{sec:lemwl}
Here we consider the general case $r\ge 1$, \ie\ $\pr(\xi=r)>0$ and $\pr(\xi<r)=0$.
Recall that, $\tilde{q}=\prod_{s=0}^{\infty} \pr(\xi \le s\vee r)$.
We have
\begin{align*}
	\var(\tilde{N}_{n}) = \sum_{t=1}^{n}\pr(\tilde{\cE}_t) + 2 \sum_{1\le s<t\le n}\pr(\tilde{\cE}_{s}\tilde{\cE}_{t}) - n^2\tilde{q}^2.
\end{align*}
We denote by $\tilde{q}_t$ the truncated product
			$\tilde{q}_t := \prod_{s=0}^{t-1} \pr(\xi \le s\vee r)$ for $t\ge 0$ which decreases to $\tilde{q}$ as $t \to \infty$.
In particular, for $t\ge r$, we have
\begin{align*}
	0\le \tilde{q}_{t}-\tilde{q} \le \tilde{q}_{t}\left(1- \prod_{s=t+1}^\infty (1-\pr(\xi > s)) \right) \le \sum_{s=t+1}^\infty \pr(\xi > s).
\end{align*}
Clearly, the events $(\tilde{\cE}_t)_{t>0}$ are identically distributed.
Moreover for $t > s$, we have
\begin{equation}
	\label{eq:EtEs}
	\tilde{\cE}_t\cap\tilde{\cE}_s= \tilde{\cE}_t \bigcap_{i=0}^{t - s-1} \{\xi_{s+i} \le i\},
\end{equation}
so that $\pr(\tilde{\cE}_t\cap\tilde{\cE}_s) =\tilde{q}\tilde{q}_{t-s}$.
In particular, we have
\begin{align}
	\label{eq:varNn}
	&\var(\tilde{N}_n)\notag\\ & = n\tilde{q} + 2\tilde{q} \sum_{t=1}^{n}(n-t)\tilde{q}_t- n^2\tilde{q}^2\notag\\
	&= n\tilde{q}(1-\tilde{q}) + 2\tilde{q} \sum_{t=1}^{n}(n-t)(\tilde{q}_t-\tilde{q}).
\end{align}
Finally we get
\begin{align*}
	\frac1{n^2} \var(\tilde{N}_n)\le \frac1n\tilde{q}(1-\tilde{q}) + \frac{2\tilde{q}}{n} \sum_{t=1}^n (\tilde{q}_t-\tilde{q}) \to 0 
\end{align*}
as $n\to\infty$. This completes the proof.

\subsection{Proof of Proposition~\ref{prop:gamma-tau}}\label{sec:propgt}
Let $h : \dZ \to \dR$ be a function such that $h(0) = 0$ and denote by
\begin{align*}
	H(x) := \sum_{i = 0}^{x}h(i),\quad x\in \dN.
\end{align*}
We can re-express $h(x) = DH(x) = H(x+1) - H(x)$, where $D$ is the discrete derivative operator. It follows from a standard result of Palm theory~\cite[Chapter 6]{RMbook} that
\begin{align*}
	\E(h(\gt_1)) = \pr(\cE_0) \E H(\gc_1)= q\E H(\gc_1),
\end{align*}
but we include a heuristic proof below for completeness. Our result then follows by picking $H(x) = x^\gb$ for $\gb > 1$.

For $t\ge 0$, define $\gl_t = \min\{\gt_i : \gt_i \ge t, i\ge 1\}$ as the next generation time after time $t$. Clearly, $\gl_{t}-t\equald \gl_{0}=\gt_1$

Fix some large integer $n$. Suppose that there are $K_{n}:=\max\{k\ge 0\mid \gt_{k}\le n\}$ many regeneration times in the interval $[0, n]$.
By stationarity, we see that:
\[
	n\E(h(\gl_0)) = \sum_{t=0}^{n-1}\E(h(\gl_t-t)).
\]
Denote by $Y_j := (\gt_{j-1}, \gt_{j}]$ the $j$-th interval in between two consecutive regeneration times after time $0$ for $j\ge 1$. For $t \in Y_j$, we have $\gl_{t}=\gt_{j}$. In particular, we have $\sum_{t \in Y_j}h(\gl_t-t) = H(\gc_j) - H(0)=H(\gc_{j})$. Thus
\[
	n\E(h(\gl_0)) = \E K_{n}\cdot \E H(\gc_1) +O(1).
\]
Finally, the result follows by scaling by $\frac1{n}$ and passing to the limit using the renewal theorem as $\E K_{n}/n\to q$.

\subsection{Proof of Proposition~\ref{prop:tau-xi}}\label{sec:proptx}
Let $F$ denote the distribution function of $\xi$. Here, we consider the case when we have $F(1)=\pr(\xi=1)>0$.
Define the positive random variable
\begin{align*}
	V := 1+\max_{i\ge 0}(\xi_i - i\vee 1)\ge 1
\end{align*}
where $\xi_{i},i\ge 0$ are i.i.d.~$\sim F$.
For any $k\ge 1$ we have
$
	\pr(V \le k) = F(k)\prod_{i=k}^{\infty} F(i).
$
It is clear that
\[
	\pr(V = 1) = q = F(1)\prod_{i=1}^{\infty} F(i).
\]
We also have that for $k\ge 0$
\begin{align}\label{eq:vhattail}
	\begin{split}
		&\pr(V > k) = 1 - F(k)\prod_{i=k}^{\infty}F(i)\\
		& \le 1 - F(k)+F(k)\sum_{i=k}^{\infty} (1 - F(i))\\
        &\le 2\E(\xi - k)_+
		\le 2\E(\xi \ind(\xi > k)).
	\end{split}
\end{align}

To analyze the distributional properties of $\gt_1$, first we note that, $\pr(\gt_1=0)=q$ and for any $\ell\ge 1$
\begin{align}\label{eq:taupr}
	\pr(\gt_1\ge \ell) = \pr( \text{there is no regeneration time in }[0,l)).
\end{align}
Fix $t_{0}:=\ell-1$ and consider the event that there is no regeneration time in $[0,l]$. Define $V_{0}:=\infty$. Here we have no restriction about the time delay graph in the interval $[t_{0}+1,t_{0}+V_{0})$. Define
\[
	V_1:= 1+\max_{0\le i < V_0}(\xi_{t_0+i} - i\vee 1) \equald V.
\]
If $t_{0}$ is not a regeneration time, we have $V_1>1$. Note that $V_1$ involves all the delay r.v.s in the interval $I_{0}:=[t_{0},t_{0}+V_{0})$. Moreover, in the time delay graph, there can be no regeneration time in the interval $[t_1+1,t_{0})$ where $t_1:=t_{0}-V_1$. We define
\[
	V_2:= 1+\max_{0\le i < V_1}(\xi_{t_1+i} - i\vee 1).
\]
If $t_1$ is not a regeneration time, we have $V_2>1$. Again, $V_2$ involves all the delay r.v.s in the interval $I_1:=[t_1,t_0)$.

Continuing this process, with $t_{n}:=t_{n-1}-V_{n}, V_{n+1}:= 1+\max_{0\le i < V_{n}}(\xi_{t_{n}+i} - i\vee 1)$ involving edges from the interval $I_{n}:=[t_{n},t_{n-1}), n\ge 1$ and so on, we see that $(V_{i})_{i\ge 0}$ is a Markov chain with $V_{0}=\infty$ and given $V_n=k\ge 1$, we have
\[
	V_{n+1}\equald 1+\max_{0\le i< k}(\xi_i - i\vee 1).
\]
The first regeneration time can be obtained when $V_{n}$ hits $1$. This process is described pictorially in Figure~\ref{fig:domination}.

\begin{figure}
	\centering
	\includegraphics[width=.8\textwidth,page=12]{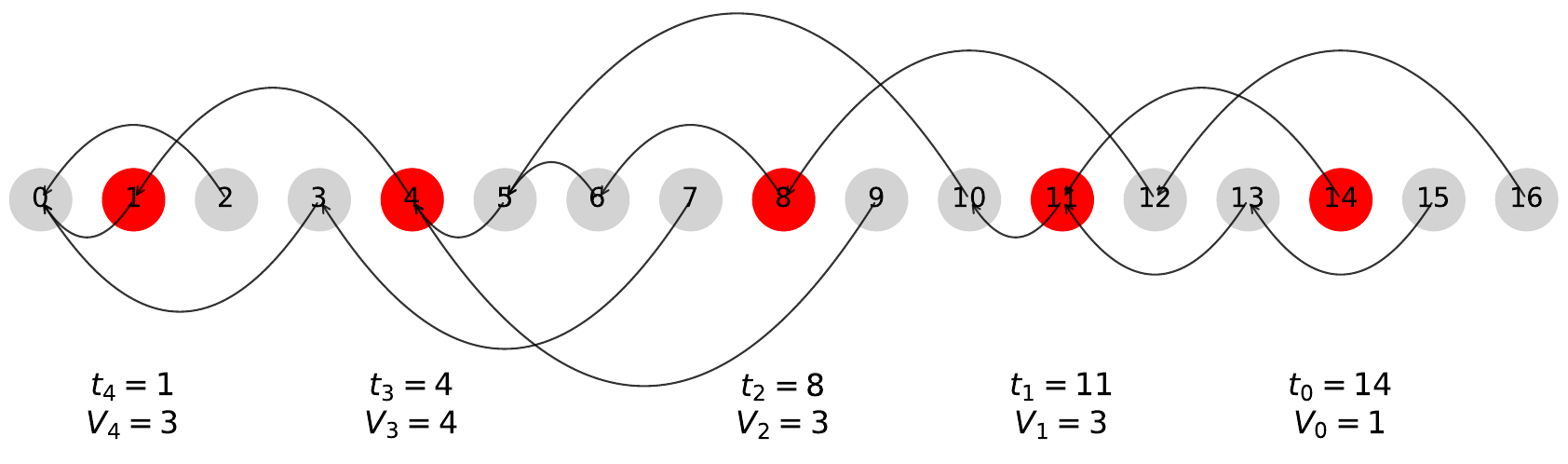}
	\caption{Stochastic domination in the proof of Proposition~\ref{prop:tau-xi} with $\ell=15,V_{i}=t_{i-1}-t_{i}, i\ge 1$.}
	\label{fig:domination}
\end{figure}

Define
\[
	N:=\inf\{n\ge 0\mid V_{n}=1\}\ge 1,
\]
the hitting time to $1$ for the Markov chain. Define the random variable
\[
	S:=\sum_{n=1}^{N} V_{n} =\sum_{n=1}^{\infty} V_{n}\ind_{n\le N} \ge 1.
\]
Thus, we have the following relation
\begin{align}\label{eq:regl}
	\{ \text{There is no regeneration point in }[0,l) \}=\{S> \ell\}.
\end{align}
In particular, using equations~\eqref{eq:taupr} and~\eqref{eq:regl}, we have $\gt_1\equald S-1$.
Thus, to control the moments of $\gt_1$ we need to get an upper bound for the moments of $S$. Note that,
\[
	\pr(V_2=1\mid V_1=k) = \prod_{i=0}^{k-1}\pr(\xi\le i\vee1) \ge q \text{ for all }k\in[1,\infty].
\]
Thus $N$ is stochastically dominated by a Geometric$(q)$ random variable.
It is also easy to see that $S$ is stochastically dominated by $\sum_{i=0}^{\widehat{N}} \widehat{V}_{i}$, where $(\widehat{V}_i)_{i\ge 1}$ are i.i.d.~random variables such that $\widehat{V}_1 \sim V \mid V > 1$ and are independent of $\widehat{N}\sim\text{Geometric}(q)$.

In particular, for $\gb>1$, we get that
\begin{align*}
	\E(\gt_1+1)^{\gb-1}=\E S^{\gb-1} \le \E\widehat{V}_1^{\gb-1} \cdot \E \widehat{N}^{\gb-1}<\infty
\end{align*}
if $\E\widehat{V}_1^{\gb-1}<\infty$ or $\E V^{\gb-1}<\infty$.

Assume that $\E\xi^{\gb}<\infty$ for some $\gb>1$.
We have using equation~\eqref{eq:vhattail} that
\begin{align*}
	&\E V^{\gb-1}
	  = \sum_{k\ge 0} ((k+1)^{\gb-1}-k^{\gb-1})\pr(V > k) \\
	 & \le 2\gb \sum_{k \ge 0} (k+1)^{\gb-2}\E(\xi \ind(\xi > k))\\& \le 2\gb\E(\xi+1)^{\gb}<\infty
\end{align*}
where the first inequality follows from equation~\eqref{eq:vhattail}. This completes the proof.

\section{Proofs of Main Results}
In this section we present the proofs of our main results.
The proof of Theorem~\ref{thm:f_2-recurrent} depends on the value of $r$; and we separate that proof into the two cases where $r = 1$ and $r > 1$.

\subsection{Proof of Theorem~\ref{thm:nakamoto-conv}}
Denote $\pi_0 = 0$, and for $k \ge 1$, we inductively define $\pi_k = \min_{n \in \dN}\{n>\pi_{k-1}: X_n \ge k\}$.
From the i.i.d.~assumption on the $(\xi_i)_i$, it is easy to see that the random variables $\chi_k:=\pi_{k+1} - \pi_k,k\ge 0$ are i.i.d..
Note that the sequence $(\chi_k)_k$ denotes the lengths of intervals where $X_t$ stays constant.
It follows from the strong law of large numbers that
\begin{align*}
	\frac{X_t}{t} \xrightarrow{a.s.} \frac1{\E(\chi_1)}.
\end{align*}

Finally, we compute the distribution of $\chi_1$ as follows.
Note that the first time of increment for $X_t$ after time $0$ is the first time when $\xi_{i} < i$.
Using the independence of $(\xi_i)_i$, we get that $\pr(\chi_1 \ge k) = \prod_{i=1}^{k}\pr(\xi_i \ge i)$; the result follows.

Observe that this proof works even when $r > 1$ because for $i \le r$, we have $\pr(\xi_i \ge i) = 1$.
The second convergence is an immediate consequence of the renewal central limit theorem.
The third convergence follows from the functional strong law of large numbers.
Finally, the last convergence follows from Donsker's theorem for renewal processes.

\subsection{Proof of Theorem~\ref{thm:nakamoto-ends}}
Here we present the proof for the general case $r\ge 1$.
Observe that if there are two regeneration intervals beginning at the instants $t_0$ and $t_0 + r$, then there are also regeneration windows beginning at all of the instants in $(t_0, t_0 + r)$. We call such a regeneration interval $[t_{0}, t_{0}+2r)$ a long regeneration interval beginning at $t_{0}$.

Fix a long regeneration interval beginning at $t_0$. Note that, almost surely, there is an infinite sequence of such regeneration intervals with finite expected inter-duration.
We consider the instant $\rho$ of the last increment of the height process $X_t$ before time $t_{0}$.
If $\rho \le t_{0} - r$, then there is an increment of $X_t$ at time $t_{*}=t_{0}$. In this case, with probability $r^{-r}$, each leaf vertex in the regeneration interval beginning at $t_{0} + r$ connects to the leaf added at time $t_{*}$, as the height process will stay constant in the interval $[t_{0},t_{0}+r)$.
From the definition of a regeneration interval, the leaf added at time $t_{*}$ will be confirmed in the asynchronous limit.
If $\rho= t_{0} - i$ for some $i=1,2,\ldots,r-1$, then there is an increment of $X_t$ at time $t_{*}=t_{0}+r-i$. The same argument can be used to show that $\pr($a long regeneration window contains a confirmed vertex$)\ge r^{-r}>0$.

It follows that the asynchronous limit exists and has infinitely many confirmed vertices, almost surely.
Since the asynchronous limit is a tree, it immediately follows that it is one-ended.

This proof reduces to a simpler argument when $r = 1$.
Indeed, recall the previous argument that if two consecutive instants are regeneration times, then the vertex added at the first instant is confirmed in the limit as $t \to \infty$.

\subsection{Proof of Theorem~\ref{thm:f_1-growth}}

First we present the proof for the $r=1$ case. The proof for the general $r$ case is essentially the same.

It is easy to see that $L_{t}=L(G_{t})$, the number of leaves at time $t$, is a non-decreasing function of $t$ with $0\le L_{t+1}-L_{t}\le 1$ a.s.~for all $t$. Here we will show that the expected number of leaves grows as $\Theta(t^{1/2})$. Since, $(\hat{L}_{k})_{k\ge 1}$ is a Markov chain, it follows that $L_{t}\to\infty$ almost surely and thus $G^{\infty}(f_{1})$ has infinitely many ends by Lemma~\ref{lem:tree-infty}.

Without loss of generality we can assume that $\cE_{0}$ holds, \ie\ $0$ is a regeneration time. Otherwise, we can shift the time to the first regeneration time $\gt_1$, which is a tight random variable.
In particular, conditional on the event that $0$ is a regeneration time, the delays $(\xi_t)_{t\ge 1}$ are independent and satisfy $(\xi_t\mid \cE_{0}) \sim (\xi\mid \xi \le t)$ for all $t\ge 1$. We use $\hat{\xi}_{t}$ for $\xi_{t}$ conditioned on $\cE_{0}$.

For the function $f_1$, we have:
\begin{align*}
	V(G_t) = V(G_{t-1}) \cup \{t\}\text{ and }E(G_t) = E(G_{t-1}) \cup \{(t, \ell_t)\},
\end{align*}
where the vertices $(\ell_t)_{t\ge 1}$ are independent and satisfy $\ell_t \sim \text{Unif}(\cL(G_{t- \hat\xi_t}))$.
We denote by
\begin{align*}
	\cI_{t, s} := \cL(G_s) \cap \cL(G_t)^c \text{ for }t \ge s,
\end{align*}
the set of leaves in $G_s$ which are not leaves in $G_t$ anymore.
We denote by $I_{t, s} := |\cI_{t, s}|$ the number of such leaves in $G_s$.
Observe that
\begin{align*}
	L_t - L_{t-1} = \ind\{\ell_t \in \cI_{t-1, t - \hat\xi_t}\}.
\end{align*}
This follows since for the function $f_1$, at most a single new leaf can be added in any time step, and the number of leaves cannot decrease at any time step.

Let $\cF_t$ be the $\sigma$-algebra generated by the delays $\hat{\xi}_1, \hat{\xi}_2, \ldots, \hat{\xi}_t$ and the leaf choices $\ell_1,\ell_2,\ldots, \ell_{t}$. We can express the conditional probability as
\begin{align}
	\pr\left(\ell_t \in \cI_{t-1, t - \hat\xi_t} \ \bigr|\ \cF_{t-1}\right)
	= \sum_{i=1}^{t}\pr(\hat\xi_t = i)\cdot \frac{I_{t-1, t-i}}{L_{t-i}}.
	\label{eq:f_1-growth-cond-prob}
\end{align}
We begin with the upper bound.
As noted above, the function $f_1$ implies that
\begin{align*}
	I_{t-1, t - i}
	 & \le (i-1) \wedge L_{t-i},\text{ and that } \\
	L_{t-1} & \le L_{t-i} - I_{t-1, t- i} + (i - 1 - I_{t-1, t- i}).
\end{align*}
Re-arranging, we get
\begin{align*}
	L_{t-i} \ge L_{t-1} + 2I_{t-1, t- i} -(i - 1).
\end{align*}
We now use the identity $\frac{ab}{b + 2a - (i-1)} \le \frac{b(i-1)}{b + (i-1)} \le i-1$ for $0\le a \le i-1 < b$ to bound equation~\eqref{eq:f_1-growth-cond-prob}.
In particular, we have
\begin{align*}
	&\pr\left(\ell_t \in \cI_{t-1, t - \hat\xi_t} \ \bigr|\ \cF_{t-1}\right)\\
	 & = \sum_{i=1}^{t}\pr(\hat\xi_t = i)\cdot \frac{I_{t-1, t- i}}{L_{t-i}} \\
	 & \le \frac1{L_{t-1}}\sum_{i=1}^{L_{t-1}}(i-1)\pr(\hat\xi_t = i) + \pr\left(\hat\xi_t-1 \ge L_{t-1} \mid \cF_{t-1}\right) \\
	 & \le \frac2{L_{t-1}}\E(\hat\xi_t - 1)_{+}.
\end{align*}
Here, the first inequality follows by breaking the interval $[1, t]$ into sub-intervals $[1, L_{t-1}]$ and $[L_{t-1} + 1, t]$; and by bounding the term ${I_{t-1, t- i}}/{L_{t-i}}$ by $(i-1)/L_{t-1}$ in the first subinterval and by $1$ in the second.
Recall from the model that $\hat{\xi}_t$ is independent of $\cF_{t-1}$, so that $\E((\hat\xi_t - 1)_{+} \mid \cF_{t-1}) = \E(\hat\xi_t - 1)_{+}$.
By taking $c := \sup_t 2\E(\hat\xi_t - 1)_{+}$, it follows that
\begin{align*}
	\pr\left(\ell_t \in \cI_{t-1, t-\hat{\xi}_t} \mid \cF_t\right) \le \frac{c}{L_{t-1}} \text{ for all }t.
\end{align*}
In particular, we have
\[
	\E(L_t \mid \cF_{t-1}) \le L_{t-1} + \frac{c}{L_{t-1}}.
\]

We conclude by examining the difference of the second moments $L_t^2 - L_{t-1}^2$.
We have that $L_t^2 - L_{t-1}^2 = (2L_{t-1} + 1)\ind\{\ell_t \in \cI_{t-1, t-1- \hat{\xi}_t}\}$, which follows from expanding $(L_{t-1}+1)^2$. Thus, in follows that
\begin{align*}
	&\E(L_t^2 \mid \cF_{t-1}) - L_{t-1}^2\\
	 & \le (2L_{t-1} + 1) \frac{c}{L_{t-1}}
	\le 2c + \frac1{L_{t-1}}
	\le 2c + 1.
\end{align*}
By induction, we have that $\E L_t^2 \le (2c + 1)t + L_0^2$. Hence by Jensen's inequality we have $\E L_t \le \sqrt{\E(L_t^2)} \le \sqrt{(2c+1)t + L_0^2}$, which establishes the upper bound for $\E L_{t}$.
We now establish a lower bound. Observe that
\begin{align*}
	\E(L_t \mid \cF_{t-1}) \ge L_{t-1} + \frac{c_1}{L_{t-1}}
\end{align*}
for some constant $c_1 > 0$.
Since $\E L_t \le \sqrt{(2c+1)t + L_0^2}$ and $\E(1/L_{t-1})\ge 1/\E L_{t-1}$, it follows by induction that
\begin{align*}
	\E L_t \ge \sum_{i = 1}^{t-1}\frac{c_1}{\sqrt{(2c+1)i + L_0^2}} = \Theta(\sqrt{t}).
\end{align*}

\subsection{Proof of Theorem~\ref{thm:f_2-recurrent}}
Let $\cA_t$ be the event that $f \neq f_1$ at time $t$.
Clearly, $(\cA_t)_{t\ge 1}$ are i.i.d.

\subsubsection{Case $1$: $r = 1$}\label{sssec:f2-r1}
Recall that $\hat{L}_{k}$ denotes the number of leaves at time $\gt_{k}$, for $k\ge 1$. Moreover, $(\hat{L}_k)_k$ is an $\dN$-valued Markov Chain. Thus, it suffices to show that it is positive recurrent.

Note that, at any time if the delay is $1$ and $f_1$ is not chosen at that time, the number of leaves goes down by at least one. Thus
\begin{align*}
	\pr( \hat{L}_2 = \ell-1 \mid \hat{L}_1 = \ell) \ge  \pr(\xi=1, \cA_1)>0
\end{align*}
and similarly $\pr( \hat{L}_2 = \ell+1 \mid \hat{L}_1 = \ell)>0$ for all $\ell\ge1$. Thus, $\hat{L}_{k}$ is irreducible.

If $(\hat{L}_k)_{k\ge 1}$ is positive recurrent, there exists a sequence of regeneration times $(\gt_{j_l})_{l\ge 1}$ such that $\hat{L}_{j_l} = 1$.
It follows that all infinite paths in $G_\infty(f)$ pass through the vertices added at the times $(\gt_{j_l})_{l\ge 1}$, which in turn establishes the result.

We show that $(\hat{L}_k)_k$ is positive recurrent using Foster's Theorem.
As the Markov chain is time homogenous, it suffices to show the following.

\begin{lemma}
	There exists $\ell_{0} < \infty$ such that $\E(\hat{L}_2 - \hat{L}_1 \mid \hat{L}_1=\ell) \le - \eps$ for some $\eps > 0$, whenever $\ell \ge \ell_{0}$.
\end{lemma}
\begin{proof}
	\ Recall that, $\gc_1=\gt_2-\gt_1$ is the gap between the first two consecutive regeneration times.
	Using the Markov structure, thus we have
	\begin{align*}
		&\E(\hat{L}_2 - \hat{L}_1 \mid \hat{L}_1=\ell)\\ & = \E(L_{\gc_1} - L_{0} \mid L_{0}=\ell, \cE_{0}) \\
		 & = \E((L_{\gc_1} - \ell)_{+} \mid L_{0}=\ell, \cE_{0}) - \E((L_{\gc_1} - L_{0})_{-} \mid L_{0}=\ell, \cE_{0}).
	\end{align*}
	We consider the two terms separately. First, we upper bound the term
	\begin{align*}
		 & \E((L_t\vee \ell - L_{t-1}\vee \ell)_{+} \mid L_{0}=\ell, \cE_{0}) \\
		 & \qquad
		= \pr(L_{t}-L_{t-1}=1, L_{t-1}\ge \ell\mid L_{0}=\ell, \cE_{0}).
	\end{align*}

	Recall from the proof of Theorem~\ref{thm:f_1-growth} that $I_{t, s}$ is the number of vertices which are leaves at time $s$ and are not leaves at time $t > s$.
	We also continue the notation from there $(\xi_{t}\mid \cE_{0})\sim \hat{\xi_t} \equald (\xi_t \mid \xi_t \le t)$; or equivalently, we use $\hat{\xi}_{t}$ when conditioning on the event that $0$ is a regeneration time.

	As in the proof of Theorem~\ref{thm:f_1-growth}, we have
	\begin{align}\label{eq:mdiff}
		\pr(L_{t} - L_{t-1} = 1 \mid \cF_{t-1}) = \sum_{i=1}^{t}\pr(\hat{\xi}_t = i)\cdot \frac{I_{t-1, t-i}}{L_{t-i}}.
	\end{align}
	Recall from the process dynamics that $L_{t-i}> L_{t-1} - i$ and that $I_{t-1, t-i} < i$.	We can substitute these bounds into~\eqref{eq:mdiff}, which yields
	\begin{align*}
		\pr(L_{t} - L_{t-1} = 1 \mid \cF_{t-1}) \le \sum_{i=1}^t \pr(\hat{\xi}_{t} = i) \cdot \left( \frac{i-1}{L_{t-1}-i+1} \wedge 1 \right).
	\end{align*}
	For any $k\in \{1,2,\ldots,t\}$, we can upper bound the rhs as
	\begin{align*}
		\frac{k-1}{L_{t-1} - k+1} + \pr(\hat{\xi}_t>k)\le \frac{k-1}{L_{t-1} - k+1} + \frac{\E\hat{\xi}_t}{k}.
	\end{align*}
	Note that $\E\hat{\xi}_t \le \sup_{s\ge 1}\E\hat{\xi}_{s}<\infty$. We can now optimize over $k$ by choosing $k$ to be the nearest integer to $\sqrt{L_{t-1}}$.
	Thus, for some universal constant $c\in (0,\infty)$, we have
	\[
		\pr(L_{t} - L_{t-1} = 1 \mid \cF_{t-1}) \le \frac{c}{\sqrt{L_{t-1}}}\text{ for all } t.
	\]

	It follows that
	\begin{align*}
		 & \E\left( (L_{t} \vee \ell - L_{t-1} \vee \ell)_+ \mid \cF_t,L_{0}=\ell, \cE_{0} \right) \\
		 & \qquad= \pr(L_{t}-L_{t-1}=1, L_{t-1}\ge \ell\mid L_{0}=\ell, \cE_{0}) \le c\;\ell^{-1/2}.
	\end{align*}

	We can now bound $\E((L_{\gc_1} - \ell)_{+} \mid L_{0}=\ell, \cE_{0})$.
	Indeed, we have
	\begin{align*}
		&\E((L_{\gc_1} - \ell)_{+} \mid L_{0}=\ell, \cE_{0})\\
		 & \le \sum_{t=1}^{\infty} \E((L_t\vee \ell - L_{t-1}\vee \ell)_{+}\cdot \ind_{\gc_1\ge t} \mid L_{0}=\ell, \cE_{0}) \\
		 & = \sum_{t=1}^{\infty} \pr(L_t - L_{t-1}=1, L_{t-1}\ge \ell, \gc_1\ge t \mid L_{0}=\ell, \cE_{0})
	\end{align*}
	Using H\"older's inequality for $p, q > 1, \frac1{p} + \frac1{q} = 1$, we get
	\begin{align*}
		\E((L_{\gc_1} - \ell)_{+} \mid L_{0}=\ell, \cE_{0})
		 & \le
		\left( c \ell^{-1/2} \right)^{1/p} \cdot \sum_{t=1}^\infty \pr(\gc_1 \ge t)^{1/q}\\
		&\le c^{1/p}\ell^{-1/(2p)}\sum_{i=1}^{\infty}\left(\E \gc_1^\gb/t^\gb \right)^{1/q},
	\end{align*}
	for $\gb > q > 1$ with $\E\gc_1^{\gb}<\infty$. Using Propositions~\ref{prop:gamma-tau} and~\ref{prop:tau-xi}, we have $\E\xi^{\gb}<\infty$ implies that $\E\gc_1^{\gb}<\infty$. Under our assumption that $\E\xi^{1+\gd}<\infty$ for some $\gd>0$, taking $\gb=1+\gd>q>1$, we get that
	\begin{align}\label{eq:ubd}
		\E((L_{\gc_1} - \ell)_{+} \mid L_{0}=\ell, \cE_{0})
		 & \le c_1\ell^{-1/(2p)}
	\end{align}
	for some constants $c_1>0, p>1$.

	Now, we bound $\E((L_{\gc_1} - \ell)_{-} \mid L_0 =\ell,\cE_{0}).$	Observe that when $\ell \ge 2$, we have
	\begin{align}\label{eq:lbd}
		\E((L_{\gc_1} - \ell)_{-} \mid L_0 =\ell,\cE_{0})
		\ge \pr(\xi_1 = 1, \cA_1) > 0
	\end{align}
	as when $f \neq f_1, \xi_1=1$ there is a decrease in the number of leaves.
	Combining the bounds~\eqref{eq:ubd} and~\eqref{eq:lbd}, for sufficiently large $\ell$, we have
	\begin{align*}
		\E(\hat{L}_2 - \hat{L}_1 \mid \hat{L}_1=\ell)
		 & \le c_1\ell^{-1/(2p)} - \pr(\xi_1 = 1, \cA_1) \le -\eps,
	\end{align*}
	for some constants $\eps > 0$. The result follows.
\end{proof}

\subsubsection{Case $2$: $r > 1$}
\label{sssec:f2-r}

We consider the Markov chain $X(t)$, given by the leaf geometry of the sequence $(G_t, G_{t+1}$, $\ldots, G_{t + r - 1})$.
Specifically, $X(t)$ takes values of $r$--tuples of finite DAGs with maximum path length $r$.
When $r = 1$, there is a natural bijection of this state space with the natural numbers $\dN$, which identifies the following analysis with the previous one.
One-endedness of the limit $G_\infty(f_2)$ follows again from the positive recurrence of $X_t$.

To see this, suppose that $L_t = r$ for some $t$ such that the interval $[t, t + r)$ is a regeneration interval.
We define an event of positive probability such that all the vertices in the regeneration interval $[t, t + r)$ satisfy the property that all vertices marked at least $t + r(r+1)$ have a path to each vertex in this regeneration interval. Hence, all vertices in this regeneration interval are confirmed in $G_\infty(f_2)$; and from the positive recurrence of $X(t)$ it follows that this event occurs infinitely often.

For convenience, we temporarily renumber the vertices in the window $[t, t + r(r+1))$ by $[1,2, \ldots, 1 + r(r+1))$. With positive probability, the function $f_1$ is not chosen in this interval.
When the number of leaves is at the minimum value $r$, the vertex number $1$ must have out-degree $1$; vertex number $2$ has out-degree $2$, and one of the vertices connected to by vertex number $2$ is also connected to by vertex number $1$.
We proceed similarly so that each of the $r$ leaves present when vertex $1$ is added is connected to by at least one of the vertices $1,2, \ldots, r$.
All other vertices $t$ connect to $t - r$ and $t - r - 1$.
It is obvious that since the last $r$ vertices correspond to a regeneration window, each of the first $r$ vertices is confirmed.
This is shown pictorially in Figure~\ref{fig:3}, for $r = 3$.
In the figure, we only draw $2$ outgoing edges for the vertices; as any $f \neq f_1$ adds at least 2 edges for each vertex whenever possible.
\begin{figure}[!ht]
    \centering
    \includegraphics[width=0.3\columnwidth,page=3]{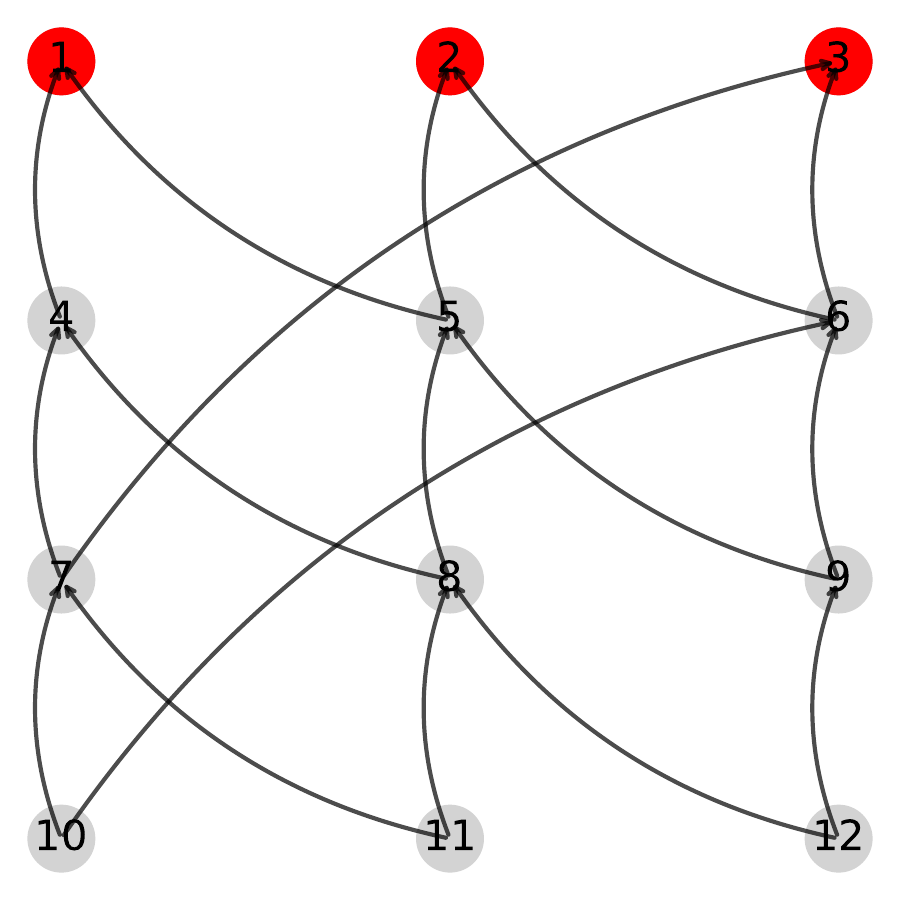}
    \caption{Confirmed blocks in a regeneration interval.}
    \label{fig:3}
\end{figure}

Thus, let $[t_i, t_i + r)$ be a sequence of disjoint regeneration intervals such that all vertices in each interval are confirmed in $G_\infty(f_2)$.
From the previous property of our event, any two infinite rays in $G_\infty(f_2)$ each pass through the regeneration intervals $[t_{2i}, t_{2i} + r)$ for all $i \in \dN$.
For $i \ge 2$, the vertices in the regeneration interval along the infinite rays each have a path to a vertex in the regeneration interval $[t_{2i - 1}, t_{2i-1} + r)$; which in turn has a path to vertices along the infinite paths contained in the regeneration interval $[t_{2(i-1)}, t_{2(i-1) + r})$.
This establishes one-endedness of the limit.

We now show that $X(t)$ is recurrent.
For a regeneration interval $[t, t + r)$, we define \[
V_{t} := \sum_{i = 0}^{r-1}\hat{L}_{t + i}.
\]
Once again, we apply Foster's theorem to get the required result. We will prove the following result.

\begin{lemma}
	For sufficiently large $\ell$, $\E(V_{\gt_2} - V_{\gt_1} \mid V_{\gt_1} =\ell) \le -\eps$ for some $\eps > 0$.
\end{lemma}

Note that there are only finitely many states such that $V_{\gt_1} < k$ for any $k$.

\begin{proof}{Proof.}
	Similarly to the proof in Subsection~\ref{sssec:f2-r1}, we express the expectation into its positive and negative components,
	\begin{align*}
	&\E(V_{\tilde{\gt}_{2}} - V_{\tilde{\gt}_1} \mid V_{\tilde{\gt}_1} =\ell)\\
        &= \E((V_{\tilde{\gt}_{2}} - V_{\tilde{\gt}_1})_{+} \mid V_{\tilde{\gt}_1} =\ell) - \E((V_{\tilde{\gt}_{2}} - V_{\tilde{\gt}_1})_{-} \mid V_{\tilde{\gt}_1} =\ell).
	\end{align*}
We now analyze each one separately.
For the positive component, the proof and conclusions in Case $1$ hold here, which can be seen by noting that for the sums in the previous proof, $(\sum_{i=r}^{t}\cdot)_{+} \le \sum_{i=1}^{t}(\cdot)_{+}$, and the same for infinite sums.

We analyze the negative component as follows.
	Observe that when $\ell \ge 2r + r$, we have
	\begin{align*}
	&\E((V_{\tilde{\gt}_{2}} - V_{\tilde{\gt}_1})_{-} \mid V_{\tilde{\gt}_1} =\ell)\\
		 &\ge \pr\bigl( \xi_r = r, \bigcap_{i=1}^r \cA_{i}, \cO \ \bigl|\  V_{0}=\ell, \tilde{\gt}_{1}=0\bigr) \ge \eps > 0.
	\end{align*}
	Here the event $\cO$ given $\tilde{\gt}_1=0$ is as follows.
	Number all the $\ell \ge 2r + r$ leaves present in $G_{0}$ by $1,2, \ldots, \ell$ in order of oldest-to-newest.
	Then $\cO$ is the event that each vertex added in the interval $[0,r)$ chooses a disjoint pair of leaves from the set $1,2, \ldots, \ell - r$.
	Notice that the probability of this event is non-decreasing in $\ell \ge 2r + r$.
	
	The remainder of this proof follows exactly as for Case $1$ in Section~\ref{sssec:f2-r1}.
\end{proof}

\subsection{Proof of Theorem~\ref{thm:state-varying}}
We observe, using the results of Theorems~\ref{thm:f_1-growth} and~\ref{thm:f_2-recurrent}, that
\begin{align*}
	\frac{c_1}{\sqrt{l}} - \frac{k\ga}{\sqrt{{l}}}
	\le \E(L_{t+1} - L_t \mid L_t = l) \le \frac{c_2}{\sqrt{l}}
	- \frac{k\ga}{\sqrt{{l}}},
\end{align*}
for appropriate positive constants $c_1 < c_2$.

Define $c'_1 = c_1 - k \ga$ and $c'_2 = c_2 - k \ga$.
Using the identity $a^{3/2} - b^{3/2} = (\sqrt{a} - \sqrt{b})(a+b+\sqrt{ab})$, we observe that
\begin{align*}
	L_{t+1}^{3/2} - L_{t}^{3/2} \approx (L_{t+1}+L_{t})\cdot \frac{L_{t+1} - L_{t}}{\sqrt{L_{t+1}} + \sqrt{L_{t}}}.
\end{align*}
Conditioning on $L_t = l$, it follows that
\begin{align*}
	c'_1 \le \E(L_{t+1}^{3/2} - L_t^{3/2} \mid L_t = l) \le c'_2.
\end{align*}
Since for any $l$, it occurs with positive probability that $L_t = l$ and $L_{t+1} = 1$, the result follows from Foster's theorem when $k\ga > c_2$; the result is obvious when $k\ga < c_1$.

\subsection{Proof of Theorem~\ref{thm:commuting-diagram}}
The convergence in time is an immediate consequence of one-endedness and we omit the proof for brevity.
The remainder of the result may be expressed as the following lemmas:

\begin{lemma}
	For any time $t \ge 0$, $\lim_{k\to\infty} G_{t}(f_k) = G_t(f_\infty)$ a.s.~when the driving sequences for each function are coupled.
\end{lemma}
\begin{proof}{Proof.}
	For any time $t$, the sequence of DAGs $(G_i(f_n))_{1 \le i \le t}$ have strictly less than $t+1$ leaves; hence the DAGs $(G_t(f_n))_{n > t}$ are all equal to $G_t(f_\infty)$; the result follows.
\end{proof}

\begin{lemma}
	$\lim_{k\to\infty} d_*(G_\infty(f_k), G_\infty(f_\infty)) = 0$ a.s.~when the driving sequences for each function are coupled.
\end{lemma}
\begin{proof}{Proof.}
	 We call a time $t$~\emph{special} if
	\begin{enumerate}
		\item $[t, t + r)$, and $[t + r, t + 2r)$ are regeneration intervals, and
		\item $|\cL(G_t(f_\infty))| = r$.
	\end{enumerate}

	From the same argument in the $r > 1$ case in the proof of Theorem~\ref{thm:f_2-recurrent}, there exist, almost surely, infinitely many special times $(\gk_i)_{i\ge 1}$ with $0 \le \gk_1 <\gk_2< \cdots$.

	Let $M_i := \min_{t \in [\gk_i, \gk_{i+1}]}|\cL(G_t(f_\infty))|$, and let $D_i$ be the hop distance between the vertices added at times $\gk_i$ and $\gk_{i+1}$.
	Both $(M_i)_i$ and $(D_i)_i$ are sequences of positive finite integer-valued random variables with infinite support. Furthermore it is clear that both $M_i , D_i \le \gk_{i+1} - \gk_{i}$.

	Assume, without loss of generality, that $0$ is a special time. For any $i$, we have that $\max_{t \le \gk_i}|\cL(G_t(f_\infty))| \le \max_{1 \le k \le i}M_k$; furthermore we have that the hop distance of the $\gk_i$-th vertex from the root is given by $D_1 +D_2+ \cdots + D_i$.

	The key step of this proof is the fact that $\max_{t \le \gk_i}|\cL(G_t(f_\infty))| \le k$ implies that for all $j \ge k$, and with coupled delays, $G_t(f_j) = G_t(f_\infty)$ for times $t \le \gk_i$.
	In particular, we have that
	\begin{align*}
		&\sup_{j \ge k}d_*(G_\infty(f_j), G_\infty(f_\infty)) \\&\le \frac1{1 + D_1 +D_2+ \cdots + D_h}\text{ whenever } \max_{1 \le l \le h}M_l \le k.
	\end{align*}
	Thus, for any $\eps>0$ and positive integer $h$, we have
	\begin{align*}
		 & \pr(\sup_{j \ge k}d_*(G_\infty(f_j), G_\infty(f_\infty))>\eps/(1+\eps)) \\
		 & \qquad\le \pr(\max_{1 \le l \le h}M_l > k) + \pr(D_1 +D_2+ \cdots + D_h <1/\eps).
	\end{align*}
 The remainder of the result is an application of concentration inequality.
	
	First, we assume that $\E\xi^2 < \infty$; this can be easily relaxed.
	We want to choose $h$ such that $h\E D_1>1/\eps$, say $h\E D_1=2/\eps$. We have, by Chebyshev's inequality
	\begin{align}\label{eq:62}
		 \pr\left(\sum_{j = 1}^h D_j < 1/\eps\right)
		&= \pr\left(\sum_{j = 1}^h (\E D_1-D_j) > \frac12h\E D_1\right) 			\notag\\&\le \frac{4\var(D_1)}{(\E D_1)^2}\cdot \frac1h,
	\end{align}
	whenever $\E D_1^2 < \infty$ and
	\begin{align*}
		\pr(\max_{1 \le l \le h}M_l > k)
        &= 1-(1-\pr(M_1>k))^{h}\\&\le h\pr(M_1>k)\le h\pr(\gk_2-\gk_1>k).
	\end{align*}
	We choose $h\approx \pr(\gk_2-\gk_1>k)^{-1/2}$ and $\eps_{k}=2/(h\E D_{1}) \approx \sqrt{\pr(\gk_2-\gk_1>k)}$ to get the bound
	\begin{align*}
	    &\pr(\sup_{j \ge k}d_*(G_\infty(f_j), G_\infty(f_\infty))>\eps_{k}/(1+\eps_{k}))\\&
	\le \text{constant}\cdot \sqrt{\pr(\gk_2-\gk_1>k)}.
		\end{align*}
	Thus the a.s.~convergence result follows when $\E D_{1}^2 < \infty$.

	Note that if two consecutive times are regeneration times, the second is special. Hence $\gk_{2} - \gk_1$ is bounded by geometric many i.i.d.~sum of $\tilde{\gc}_{i}$'s.
	Clearly, the  condition $\E D_{1}^2 < \infty$ holds when $\E(\gk_2 - \gk_1)^2 < \infty$.
	From Propositions~\ref{prop:gamma-tau} and~\ref{prop:tau-xi}, this holds when $\pr(\xi = 1) > 0$ and $\E\xi^2 < \infty$.

	More generally, if $\E\xi^{1 + \delta} < \infty$ for some $\delta > 0$, we have $\E(\gk_2 - \gk_1)^{1+\gd} < \infty$. Moreover, we can obtain a bound of the order $O(h^{-\delta})$ in equation~\eqref{eq:62} and the rest of this proof follows by taking $h\approx \pr(\gk_2-\gk_1>k)^{-1/(1+\gd)}$.
	Note that for convergence in distribution, we only require finiteness of the first moment of $\xi$.
\end{proof}

\bibliographystyle{plain}
\bibliography{acm}

\end{document}